\documentclass[12]{amsart}
\usepackage{amsmath,amssymb}
\usepackage{amsthm,color,enumerate,comment,centernot,enumitem,url,cite}
\usepackage{graphicx,relsize,bm}
\usepackage{mathtools}
\usepackage{array}
\usepackage{enumitem}
\setenumerate[0]{label=(\alph*)}

\makeatletter
\newcommand{\tpmod}[1]{{\@displayfalse\pmod{#1}}}
\newcommand{\Mod}[1]{\ (\mathrm{mod}\enspace #1)}
\makeatother

\newcommand{\ord}{\operatorname{ord}}

\newtheorem{thm}{Theorem}[section]
\newtheorem{lemma}[thm]{Lemma}

\newtheorem{cor}[thm]{Corollary}

\theoremstyle{definition}
\newtheorem{ex}[thm]{Example}

\theoremstyle{remark}

\theoremstyle{definition}
    \newtheorem{defn}[thm]{Definition}

\newtheorem{rem}[thm]{Remark}

\newcommand{\abs}[1]{\left|{#1}\right|}

\def\Z {{\mathbb Z}}

\def\NN {{\mathcal N}}

\def\Q {{\mathbb Q}}

\def\C {{\mathcal C}}

\def\D {{\mathcal D}}

\def\U {{\Upsilon}}

\def\F {{\mathbb F}}
\def\D {{\mathcal D}}
\def\Z {{\mathbb Z}}
\def\Q {{\mathbb Q}}
\def\C {{\mathbb C}}

\def\CC {{\mathcal C}}
\def\Gal{{\mbox {Gal} }}

\makeatletter
\@namedef{subjclassname@2020}{%
  \textup{2020} Mathematics Subject Classification}
\makeatother

\def\red#1 {\textcolor{red}{#1 }}
\def\blue#1 {\textcolor{blue}{#1 }}

\numberwithin{equation}{section}

\def\Z {{\mathbb Z}}
\begin{document}

\title[Power-Compositional Characteristic Polynomials]{The Monogenicity of Power-Compositional Characteristic Polynomials}


\author{Lenny Jones}
\address{Professor Emeritus, Department of Mathematics, Shippensburg University, Shippensburg, Pennsylvania 17257, USA}
\email[Lenny~Jones]{doctorlennyjones@gmail.com}

\date{\today}

\begin{abstract} Let $f(x)\in \Z[x]$ be monic with $\deg(f)=N\ge 2$. Suppose that $f(x)$ is monogenic, and that $f(x)$ is the characteristic polynomial of the $N$th order linear recurrence sequence $\Upsilon_f:=(U_n)_{n\ge 0}$ with initial conditions
 \[U_0=U_1=\cdots =U_{N-2}=0 \quad \mbox{and} \quad U_{N-1}=1.\] Let $p$ be a prime such that $f(x)$ is irreducible over $\F_p$ and $f(x^p)$ is irreducible over $\Q$. We prove that $f(x^p)$ is monogenic if and only if $\pi(p^2)\ne \pi(p)$, where $\pi(m)$ denotes the period of $\Upsilon_f$ modulo $m$. These results extend previous work of the author, and provide a new and simple test for the monogenicity of $f(x^p)$. We also provide some infinite families of such polynomials. This article extends previous work of the author.
   \end{abstract}


\maketitle
\section{Introduction}\label{Section:Intro}

For a monic polynomial $f(x)\in {\mathbb Z}[x]$ of degree $N\ge 2$ that is irreducible over ${\mathbb Q}$, we say that $f(x)$ is \emph{monogenic} if  $\Theta=\{1,\theta,\theta^2,\ldots ,\theta^{N-1}\}$ is a basis for the ring of integers ${\mathbb Z}_K$ of $K={\mathbb Q}(\theta)$, where $f(\theta)=0$. Such a basis $\Theta$ is called a \emph{power basis}.
Since \cite{Cohen}
\begin{equation} \label{Eq:Dis-Dis}
\Delta(f)=\left[\Z_K:\Z[\theta]\right]^2\Delta(K),
\end{equation}
 where $\Delta(f)$ and $\Delta(K)$ denote, respectively, the discriminants over $\Q$ of $f(x)$ and $K$, we see that $f(x)$ is monogenic if and only if $\left[\Z_K:\Z[\theta]\right]=1$ or, equivalently, $\Delta(f)=\Delta(K)$. Thus, a sufficient, but not necessary, condition for the monogenicity of $f(x)$ is that $\Delta(f)$ is squarefree. Also, to be clear, it is certainly possible that there exists a power basis for $\Z_K$ despite the fact that $f(x)$ is not monogenic.

 Let $N$ and $a_0, a_1,\ldots ,a_N$ be integers with $N\ge 2$. Define
  $\U_f:=(U_n)_{n\ge 0}$ to be the $N$th order linear recurrence sequence such that
 \begin{align}\label{Eq:Upsilon}
 \begin{split}
 U_0&=U_1=\cdots =U_{N-2}=0, \quad U_{N-1}=1 \quad \mbox{and}\\
 U_n&=a_{1}U_{n-1}+a_{2}U_{n-2}+\cdots +a_{N-1}U_{n-N+1}+a_NU_{n-N} \quad \mbox{for all $n\ge N$,}
 \end{split}
 \end{align} so that
  \begin{equation}\label{Eq:f}
 f(x):=x^N-a_{1}x^{N-1}-a_{2}x^{N-2}-\cdots -a_{N-1}x-a_N,
 \end{equation} is the characteristic polynomial of $\U_f$. It is well known that $\U_f$ is periodic modulo any integer $m\ge 2$ with $\gcd(m,a_N)=1$,
 and we denote the length of the period as $\pi(m)$.

\begin{defn}\label{Def:p-irreducible}
  Given a monic polynomial $f(x)\in \Z[x]$ and a prime $p$, we say that \emph{$f(x)$ is $p$-irreducible} if
  \[f(x) \mbox{ is irreducible over } \F_p \quad \mbox{and} \quad f(x^p)  \mbox{ is irreducible over } \Q.\]
\end{defn}
Note that if $f(x)$ is $p$-irreducible, then $\gcd(p,f(0))=1$. Consequently, if $f(x)$ is $p$-irreducible and $f(x)$ is the characteristic polynomial for $\U_f$, as defined in \eqref{Eq:Upsilon}, then $\U_f$ is periodic modulo $p$.


 \begin{defn}\label{Def:U-prime}
Let $\U_f$ and $f(x)$ be as defined in \eqref{Eq:Upsilon} and \eqref{Eq:f}. Suppose that $p$ is a prime such that $f(x)$ is $p$-irreducible. We say that $p$ is an \emph{$\U_f$-prime} if $\pi(p^2)=\pi(p)$. 
\end{defn}
In this article, we prove the following:
\begin{thm}\label{Thm:Main1}
Let $\U_f$ and $f(x)$ be as defined in \eqref{Eq:Upsilon} and \eqref{Eq:f}, with $f(x)$ monogenic. Let $p$ be a prime such that $f(x)$ is $p$-irreducible. Then $f(x^p)$ is monogenic if and only if $\pi(p^2)\ne \pi(p)$.
  \end{thm}
\begin{rem}
  Theorem \ref{Thm:Main1} extends previous work of the author \cite{JonesTJM, JonesFACM, JonesAJM}.
\end{rem}
  Theorem \ref{Thm:Main1} provides us with the following new and simple test for the monogenicity of $f(x^p)$ when $f(x)$ is monogenic and $p$-irreducible.
  \begin{cor}\label{Cor:Main1}
    Let $\U_f$ and $f(x)$ be as defined in \eqref{Eq:Upsilon} and \eqref{Eq:f}, with $f(x)$ monogenic. Let $p$ be a prime such that $f(x)$ is $p$-irreducible. Then $f(x^p)$ is monogenic if and only if $\pi(p^2)\equiv 0 \pmod{p}$.
  \end{cor} 
 As an example to illustrate the concepts presented here, suppose that $\U_f$ is defined as in \eqref{Eq:Upsilon}, such that
 \[f(x)=x^4+21x^3+86x^2+21x+1\] is the characteristic polynomial of $\U_f$. Using Maple, we can easily confirm that $f(x)$ is monogenic, $f(x)$ is  $p$-irreducible for $p=37$, and that
 \[\pi(37^2)=\pi(37)=137\not \equiv 0 \Mod{37}.\] Thus, $p=37$ is an $\U_f$-prime and $f(x^{37})$ is not monogenic by Corollary \ref{Cor:Main1}.

\section{Preliminaries}\label{Section:Prelim}

The following theorem, known as \emph{Dedekind's Index Criterion}, or simply \emph{Dedekind's Criterion} if the context is clear, is a standard tool used in determining the monogenicity of a monic polynomial that is irreducible over $\Q$.
\begin{thm}[Dedekind \cite{Cohen}]\label{Thm:Dedekind}
Let $K=\Q(\theta)$ be a number field, $T(x)\in \Z[x]$ the monic minimal polynomial of $\theta$, and $\Z_K$ the ring of integers of $K$. Let $q$ be a prime number and let $\overline{ * }$ denote reduction of $*$ modulo $q$ (in $\Z$, $\Z[x]$ or $\Z[\theta]$). Let
\[\overline{T}(x)=\prod_{i}\overline{t_i}(x)^{e_i}\]
be the factorization of $T(x)$ modulo $q$ in $\F_q[x]$, and set
\[g(x)=\prod_{i}t_i(x),\]
where the $t_i(x)\in \Z[x]$ are arbitrary monic lifts of the $\overline{t_i}(x)$. Let $h(x)\in \Z[x]$ be a monic lift of $\overline{T}(x)/\overline{g}(x)$ and set
\[F(x)=\dfrac{g(x)h(x)-T(x)}{q}\in \Z[x].\]
Then
\[\left[\Z_K:\Z[\theta]\right]\not \equiv 0 \pmod{q} \Longleftrightarrow \gcd\left(\overline{F},\overline{g},\overline{h}\right)=1 \mbox{ in } \F_q[x].\]
\end{thm}

The next two theorems are due to Capelli \cite{S}.
 \begin{thm}\label{Thm:Capelli1}  Let $f(x)$ and $h(x)$ be polynomials in $\Q[x]$ with $f(x)$ irreducible. Suppose that $f(\alpha)=0$. Then $f(h(x))$ is reducible over $\Q$ if and only if $h(x)-\alpha$ is reducible over $\Q(\alpha)$.
 \end{thm}

\begin{thm}\label{Thm:Capelli2}  Let $c\in \Z$ with $c\geq 2$, and let $\alpha\in\C$ be algebraic.  Then $x^c-\alpha$ is reducible over $\Q(\alpha)$ if and only if either there is a prime $p$ dividing $c$ such that $\alpha=\gamma^p$ for some $\gamma\in\Q(\alpha)$ or $4\mid c$ and $\alpha=-4\gamma^4$ for some $\gamma\in\Q(\alpha)$.
\end{thm}

The following theorem (due to John Cullinan \cite{Cullinan}) gives the formula for the discriminant of the composition of two polynomials. For a proof, see \cite{HJmonocyclo}.
\begin{thm}\label{Thm:Disccomp}
Let $f(x)$ and $g(x)$ be polynomials in $\mathbb{Q}[x]$, with respective leading coefficients $a$ and $b$, and respective degrees $m$ and $n$.  Then
$$\Delta(f\circ g)=(-1)^{m^2n(n-1)/2}\cdot a^{n-1}b^{m(mn-n-1)}\Delta(f)^{n}Res(f\circ g,g'),$$ where $Res$ is the resultant.
\end{thm}

The next theorem follows from Corollary (2.10) in \cite{Neukirch}.
\begin{thm}\label{Thm:CD}
  Let $L$ and $K$ be number fields with $L\subset K$. Then \[\Delta(L)^{[K:L]} \bigm|\Delta(K).\]
\end{thm}

\section{The Proofs of Theorem \ref{Thm:Main1} and Corollary \ref{Cor:Main1}}\label{Section:Main1}
We first prove a lemma. 

  \begin{lemma}\label{Lem:Main1}
  Let $\U_f$ and $f(x)$ be as defined in \eqref{Eq:Upsilon} and \eqref{Eq:f}. Let $p$ be a prime such that $f(x)$ is $p$-irreducible,
     and suppose that $f(\alpha)=0$ with $\alpha\in \F_{p^N}$. 
    Then,
 \begin{enumerate}
 \item \label{Per:I0} the zeros of $f(x)$ in $\F_{p^N}$ are $r_j:=\alpha^{p^j}$ with $j\in J:=\{0,1,2,\ldots ,N-1\}$,
 and $\ord_p(r_j)=\pi(p)$,
    \item \label{Per:I1} modulo $p^2$, $f(x)$ has exactly $N$ distinct zeros $s_j$ with $j\in J$, and $\ord_{p^2}(s_j)=\pi(p^2)$,
    \item \label{Per:I2}  $p^N\equiv 1\Mod{\pi(p)}$,
    \item \label{Per:I3} $\pi(p^2)\in \{\pi(p),p\pi(p)\}$,
    \item \label{Per:I4} furthermore, if $\pi(p^2)=\pi(p)$, then the distinct zeros $s_j$ of $f(x)$ modulo $p^2$ are
    $\alpha^{p^j} \Mod{p^2}$. 
  \end{enumerate}
\end{lemma}
\begin{proof}
Since the Frobenius automorphism generates $\Gal_{\F_p}(f)$, we have that the zeros of $f(x)$ in $\F_p(\alpha)$ are $\alpha^{p^j} \Mod{p}$ for $j\in J$. Consequently,
\begin{align}\label{Eq:constant term}
\begin{split}
  \alpha^{p^N-1}&\equiv \left(\alpha^{(p^N-1)/(p-1)}\right)^{p-1}\Mod{p}\\
  &\equiv \left(\alpha^{p^{N-1}+p^{N-2}+\cdots +p+1}\right)^{p-1}\Mod{p}\\
  &\equiv \left(\prod_{j\in J}\alpha^{p^j}\right)^{p-1} \Mod{p}\\
  &\equiv \left((-1)^{N+1}a_N\right)^{p-1} \Mod{p}\\
  &\equiv 1 \Mod{p}.
  \end{split}
\end{align}
    From \cite{Robinson}, we have that the order, modulo $m\in \{p,p^2\}$, of the companion matrix $\CC$ for the characteristic polynomial $f(x)$ of $\U_f$ is $\pi(m)$. Thus, item \ref{Per:I0} follows from the fact that the eigenvalues of $\CC$ are precisely $\alpha^{p^j}$ for $j\in J$. Since $f(x)$ is $p$-irreducible, the zeros of $f(x)$ modulo $p^2$ are precisely the unique Hensel lifts $s_j$ of $r_j$. Hence, it follows from \cite{Robinson} that  $\ord_{p^2}(s_j)=\pi(p^2)$, which establishes item \ref{Per:I1}.
    Item \ref{Per:I2} follows from item \ref{Per:I0} and \eqref{Eq:constant term}, while item \ref{Per:I3} follows from \cite{Robinson}. If $\pi(p^2)=\pi(p)$, then $\alpha^{\pi(p)}\equiv 1 \pmod{p^2}$, and the zeros of $f(x)$ modulo $p^2$ are simply $\alpha^{p^j} \Mod{p^2}$ with $j\in J$, which are the Hensel lifts of $\alpha^{p^j} \Mod{p}$ in this case, and item \ref{Per:I4} is established.
  \end{proof}

We are now in a position to present the proof of Theorem \ref{Thm:Main1}.
\begin{proof}[Proof of Theorem \ref{Thm:Main1}]
Let $g(x)=x^p$ so that $(f\circ g)(x)=f(x^p)$. From Theorem \ref{Thm:Disccomp}, we have that
\begin{equation}\label{Eq:Delfcircg} \abs{\Delta(f\circ g)}=\abs{p^{pN}\Delta(f)^p}.
\end{equation} 
Let $K=\Q(\theta^p)$, where $f(\theta^p)=0$, and let $L=\Q(\theta)$. Let $\Z_L$ denote the ring of integers of $L$.
Since $f(x)$ is monogenic, we have that $\Delta(K)=\Delta(f)$. Thus, it follows from Theorem \ref{Thm:CD} that
\[\Delta(f)^p \mbox{ divides } \abs{\Delta(L)}=\dfrac{\abs{\Delta(f\circ g)}}{[\Z_L:\Z[\theta]]^2},\]
which implies, from  \eqref{Eq:Delfcircg}, that
\[[\Z_L:\Z[\theta]]^2 \mbox{ divides } \abs{\dfrac{\Delta(f\circ g)}{\Delta(f)^p}}=p^{pN}.\]
Hence, the monogenicity of $f(x^p)$ is completely determined by the prime $p$.

 We apply Theorem \ref{Thm:Dedekind} to $T(x):=f(x^p)$ using the prime $p$.
 Since $f(x)$ is irreducible in $\F_p[x]$, we have that $\overline{f}(x^p)=f(x)^p$, and
 we can let
\[g(x)=f(x) \quad \mbox{and}\quad h(x)=f(x)^{p-1}.\]
 Then
     \[pF(x)=g(x)h(x)-T(x)=f(x)^{p}-f(x^p),\] so that
     \begin{equation}\label{Eq:pF}
     pF(\alpha)\equiv f(\alpha)^p-f(\alpha^p)\equiv  -f(\alpha^p) \Mod{p^2}.
     \end{equation}
     Since $f(x)$ is irreducible in $\F_p[x]$, we conclude that $\gcd(\overline{f},\overline{F})\in \{1,\overline{f}\}$ in $\F_p[x]$. Hence, it follows from Theorem \ref{Thm:Dedekind} and \eqref{Eq:pF} that
     \begin{align*}
      [\Z_{L}:\Z[\theta]]\equiv 0\Mod{p}&\ \Longleftrightarrow \ \gcd(\overline{f},\overline{F})=\overline{f}\\
       &\ \Longleftrightarrow \ F(\alpha)\equiv 0 \Mod{p}\\
       &\ \Longleftrightarrow \ f(\alpha^p)\equiv 0 \Mod{p^2}.
     \end{align*}
     Thus,
     \begin{equation}\label{Eq:Firstequivalency}
     f(x^p) \mbox{ is monogenic if and only if } f(\alpha^p)\not \equiv 0 \Mod{p^2}.
     \end{equation}

          As in Lemma \ref{Lem:Main1}, let $r_j=\alpha^{p^j}$ with $j\in J$, be the zeros of $f(x)$ in $\F_{p^N}$, and let $s_j$ be the unique Hensel lift modulo $p^2$ of $r_j$. Note that $r_0=s_0=\alpha$.

          Suppose that $\pi(p^2)\ne \pi(p)$, so that $\pi(p^2)=p\pi(p)$ by Lemma \ref{Lem:Main1}. Then $\ord_{p^2}(s_j)=p\pi(p)$ for all $j\in J$ by Lemma \ref{Lem:Main1}. In particular, we have that $\ord_{p^2}(s_0)=\ord_{p^2}(s_1)=p\pi(p)$. If $f(\alpha^{p})\equiv 0 \Mod{p^2}$, then $s_1=r_1$, and we arrive at the contradiction \[\ord_{p^2}(s_1)=\ord_{p^2}(\alpha^p)=\ord_{p^2}(\alpha)/p=\pi(p).\] Hence, $f(\alpha^{p})\not \equiv 0 \Mod{p^2}$.

          Conversely, if $\pi(p^2)=\pi(p)$, then $f(\alpha^p)\equiv 0 \pmod{p^2}$ by Lemma \ref{Lem:Main1}. Hence, we have that
          \begin{equation}\label{Eq:Secondequivalency}
            \pi(p^2)\ne \pi(p) \mbox{ if and only if } f(\alpha^p)\not \equiv 0 \Mod{p^2}.
          \end{equation}
The theorem then follows by combining \eqref{Eq:Firstequivalency} and \eqref{Eq:Secondequivalency}.
          \end{proof}

\begin{proof}[Proof of Corollary \ref{Cor:Main1}]
By item \ref{Per:I2} of Lemma \ref{Lem:Main1}, we have that $\pi(p)\not \equiv 0 \pmod{p}$. Hence, the corollary follows from Theorem \ref{Thm:Main1} and item \ref{Per:I3} of Lemma \ref{Lem:Main1}.
\end{proof}

\section{Infinite Families Illustrating Theorem \ref{Thm:Main1}}
Although we only require a special case of the following lemma, we provide a proof in more generality because the lemma is of some interest in its own right.
\begin{lemma}\label{Lem:FU}
   Let $f(x)\in \Z[x]$ be irreducible over $\Q$, and let $m\ge 1$ be an integer. Suppose that $f(\varepsilon)=0$ and that $\varepsilon:=\lambda_0$ is a fundamental unit of $K=\Q(\varepsilon)$. Then $f(x^m)$ is irreducible over $\Q$.
  \end{lemma}
  \begin{proof}
  Let
  \[\Lambda=\{\varepsilon, \lambda_1,\lambda_2,\ldots ,\lambda_n\}\]
  be a system of fundamental units for $\Z_K$, where $\Z_K$ is the ring of integers of $K$, so that the rank of the unit group in $\Z_K$ is $n+1$. We let $\NN$ denote the algebraic norm $\NN_{K/\Q}$.
       Assume, by way of contradiction, that $f(x^m)$ is reducible over $\Q$. Then, since $f(x)$ is irreducible over $\Q$, we have by Theorem \ref{Thm:Capelli1} and Theorem \ref{Thm:Capelli2} that, for some $\gamma\in K$,  either $\varepsilon=\gamma^p$ for some prime $p$ dividing $m$, or $4\mid m$ and $\varepsilon=-4\gamma^4$. We see easily that the case $\varepsilon=-4\gamma^4$ is impossible since $\NN(-4\gamma^4)\ne \pm 1$.
        So suppose that $\varepsilon=\gamma^p$. Then
  \[\NN(\gamma)^p=\NN(\gamma^p)=\NN(\varepsilon)=\pm 1,\] which implies that $\NN(\gamma)=\pm 1$. 
  Therefore, $\gamma$ is a unit and
  \[\gamma=\pm \varepsilon^e\prod_{i=1}^n\lambda_i^{e_i},\] for some $e,e_i\in \Z$.
   Consequently,
  \[\varepsilon=\gamma^p=(\pm 1)^p\varepsilon^{ep}\prod_{i=1}^n\lambda_i^{e_ip},\] which implies that
  \[(\pm 1)^p\varepsilon^{ep-1}\prod_{i=1}^n\lambda_i^{e_ip}=1.\] Hence, since the elements of $\Lambda$ are multiplicatively independent, it follows that $ep-1=0$, which is impossible.
  \end{proof}

\begin{ex}\label{Ex:Quadratics}
  Let $k\ge 1$ be an integer, and let $f(x)=x^2-kx-1$.\\  We assume that $k\not \equiv 0 \Mod{4}$ and that $\D:=(k^2+4)/\gcd(2,k)^2$ is squarefree. Although these polynomials were investigated in \cite{JonesTJM}, we provide a sketch here. It is easy to see that there are infinitely many values of $k$ for which these conditions hold. Then $f(x)$ is irreducible over $\Q$ and $f(\varepsilon)=0$, where $\varepsilon=(a+\sqrt{a^2+4})/2$. Since $\D$ is squarefree, we conclude that $f(x)$ is monogenic if $k\equiv 1 \pmod{2}$, and the monogenicity of $f(x)$ when $k\equiv 0 \pmod{2}$ is determined completely by the prime $p=2$. Applying Theorem \ref{Thm:Dedekind} with $T(x):=f(x)$ and $p=2$ in the case when $k\equiv 0 \pmod{2}$, we can let
  \[g(x)=(x-1)(x+1) \quad \mbox{and} \quad h(x)=1.\]
  Then
   \[T(x)=\dfrac{g(x)h(x)-T(x)}{2}=\dfrac{(x^2-1)-(x^2-kx-1)}{2}=(k/2)x,\]
   so that $\gcd(\overline{T},\overline{g})=1$ since $k\not \equiv 0 \pmod{4}$. Thus, $f(x)$ is monogenic. It follows from \cite{Y} that $\varepsilon=(a+\sqrt{a^2+4})/2$ is the fundamental unit of $\Q(\sqrt{\D})$. Therefore, we deduce from Lemma \ref{Lem:FU} that $f(x^p)$ is irreducible over $\Q$ for all primes $p$. Since there are infinitely many primes $p$ for which $f(x)$ is irreducible in $\F_p[x]$ by the Chebotarev density theorem (or the simpler theorem of Frobenius) \cite{Jan}, there are infinitely many primes $p$ for which $f(x)$ is $p$-irreducible. It is then clear that Theorem \ref{Thm:Main1} applies to these polynomials. To illustrate Theorem \ref{Thm:Main1} for such polynomials, a computer search reveals that the only non-monogenic polynomials $f(x^p)$ (so that $\pi(p^2)=\pi(p)$),  with $1\le k\le 50$ and $p\le 11$, are given in the following table:
     \begin{table}[h]
 \begin{center}
\begin{tabular}{ccc}
 $k$ &  $p$ & $\pi(p^2)=\pi(p)$\\ \hline \\[-8pt]
 $\{5,13,22,23,24,31,32,41,49,50\}$ & 3 & 8\\ [2pt]
 $\{7,8,18,43,44\}$ & 5 & 12\\ [2pt]
 $\{25\}$ & 7 & 16\\ [2pt]
$\{19,20\}$ & 11 & 8\\ [2pt]
$\{5\}$ & 11 & 24\\ [2pt]
 \end{tabular}
\end{center}
\caption{Non-monogenic $f(x^p)$ and the corresponding $\Upsilon_f$-primes $p$}
 \label{T:1}
\end{table}
\begin{rem}
  The $\Upsilon_f$-primes in Table \ref{T:1} are also known as \emph{$k$-Wall-Sun-Sun primes}, or simply \emph{Wall-Sun-Sun primes} when $k=1$. The existence of Wall-Sun-Sun primes is still an open question.
\end{rem}
\end{ex}

\begin{ex}\label{Ex:Shanks}
  Let $k\ge 1$ be an integer, and let $f(x)=x^3-kx^2-(k+3)x-1$.\\ These polynomials are known as the ``simplest cubics", and were studied extensively by Daniel Shanks \cite{Shanks}. We assume that $k\not \equiv 3 \Mod{9}$ and $(k^2+3k+9)/\gcd(3,k)^2$ is squarefree. Power-compositional Shanks polynomials were investigated in \cite{JonesFACM}, and we omit the details since they are similar to Example \ref{Ex:Quadratics}.
\end{ex}

\begin{ex}\label{Ex:AO}
  Let
  \[f(x)=x^4-2(2k+1)x^3+((2k+1)^2-d+2)x^2-2(2k+1)x+1,\] where $k\in \Z$ and $d\in \{-2,2\}$
  such that
  \[\D:=(4k^2-4k-d+1)(4k^2+12k-d+9)\] is squarefree. Note that, for each choice of $d$, there exist infinitely many values of $k$ such that $\D$ is squarefree \cite{BB} and $\D\ge 1024$. For such values of $d$ and $k$, let $\alpha=2k+1+\sqrt{d}$. Observe that $\alpha\not \in \Z$ so that $\alpha^2-4$ is not a square in $\Z$. Under these conditions, Aoki and Kishi \cite{AK} proved that $f(x)$ is irreducible over $\Q$, $\Gal_{\Q}(f)\simeq D_4$ and $\varepsilon$ is a fundamental unit of $K=\Q(\varepsilon)$, where $f(\varepsilon)=0$. It then follows from Lemma \ref{Lem:FU} that $f(x^m)$ is irreducible over $\Q$ for any integer $m\ge 1$. Consequently, $f(x)$ is $p$-irreducible for any prime $p$ for which $f(x)$ is irreducible over $\F_p$; and since $D_4$ has an element of order 4, there exist infinitely many such primes $p$ by the Chebotarev density theorem.

  Next, we address the monogenicity of $f(x)$. 
  Let $\Z_K$ denote the ring of integers of $K=\Q(\varepsilon)$. A computer calculation reveals that $\Delta(f)=64\D$. Since $\D$ is squarefree, it follows that the monogenicity of $f(x)$ is completely determined by the prime $p=2$. We use Theorem \ref{Thm:Dedekind} with $T(x):=f(x)$ and $p=2$. Then
  $\overline{T}(x)=(x^2+x+1)^2$, and we can let $g(x)=h(x)=x^2+x+1$. Thus,
  \begin{align*}
  F(x)&=\dfrac{g(x)h(x)-T(x)}{2}\\
  &=\dfrac{(x^2+x+1)^2-f(x)}{2}\\
  &=(2k+2)x^3-(2k^2+2k+d/2)x^2+(2k+2)x,
  \end{align*}
  which implies that $\overline{F}(x)=x^2$. Hence, $\gcd(\overline{g},\overline{F})=1$, and we conclude from Theorem \ref{Thm:Dedekind} that
   \begin{equation}\label{Eq:Condition}
  [\Z_K:\Z(\varepsilon)]\not \equiv 0 \Mod{2}.
  \end{equation} Consequently, $f(x)$ is monogenic.

  Thus, we have shown that $f(x)$ represents an infinite family of polynomials satisfying the hypotheses of Theorem \ref{Thm:Main1}. For such polynomials, a computer search reveals that the only non-monogenic polynomials $f(x^p)$,  with $-100\le k\le 100$ and $p\le 97$, are indicated in the following table:
  \begin{table}[h]
 \begin{center}
\begin{tabular}{crcc}
 $k$ & $d$ & $p$ & $\pi(p^2)=\pi(p)$\\ \hline \\[-8pt]
$\{-70,-21\}$ & $-2$ & 7 & 25\\ [2pt]
$\{20,69\}$ & $-2$ & 7 & 50\\ [2pt]
$\{k\equiv 8 \Mod{25}\}$ & $2$ & 5 & 13\\ [2pt]
$\{k\equiv 16 \Mod{25}\}$ & $2$ & 5 & 26\\ [2pt]
 \end{tabular}
\end{center}
\caption{Non-monogenic $f(x^p)$ and the corresponding $\Upsilon_f$-primes $p$}
 \label{T:2}
\end{table}
\end{ex}

\begin{ex}\label{Ex:Schopp}
  Let
  \begin{equation}\label{Eq:FamSchopp}
  f(x)=x^4-kx^3+b(k-1)x^2+2b^2x-b^3,
  \end{equation} where $k\in \Z$ and $b\in \{-1,1\}$, with $k\le 3$ when $b=1$ and $k\le -5$ when $b=-1$, 
  such that
  \[\D:=(k^2+4b)(4k-16b+1)\] is squarefree (We note that, for each choice of $b$, there are infinitely many values of $k$ such that $\D$ is squarefree  \cite{BB}). It is shown in \cite{LPS} that $f(x)$ is irreducible over $\Q$,
  $\varepsilon$ is a fundamental unit of $K:=\Q(\varepsilon)$ where $f(\varepsilon)=0$, and that $\Gal_{\Q}(f)\simeq D_4$, with $L:=\Q(\sqrt{k^2+4b})$ a quadratic subfield of $K$.  Thus, $f(x^p)$ is irreducible over $\Q$ for every prime $p$
  by Lemma \ref{Lem:FU}, and there exist infinitely many primes $p$ for which $f(x)$ is irreducible in $\F_p[x]$ by the Chebotarev density theorem.  Consequently, for every prime $p$ for which $f(x)$ is irreducible in $\F_p[x]$, we have that $f(x)$ is $p$-irreducible.

  We show next that $f(x)$ is monogenic. Let $\Z_K$ denote the ring of integers of $K$. A computer calculation gives
  \[\Delta(f)=b^6(k^2+4b)^2(4k-16b+1)=(k^2+4b)^2(4k-16b+1).\] Since $\D$ is squarefree, we see that $k\equiv 1\pmod{2}$, and so $\Delta(L)=k^2+4b$.
  Therefore, $(k^2+4b)^2$ divides $\Delta(K)$ by Theorem \ref{Thm:CD}. Since $4k-16b+1$ is squarefree, it follows from \eqref{Eq:Dis-Dis} that $\Delta(f)=\Delta(K)$ and $f(x)$ is monogenic.

  Therefore, Theorem \ref{Thm:Main1} applies to the families in \eqref{Eq:FamSchopp}. A computer search shows that $f(x^p)$ is monogenic when $-5000<k\le -5$ and $p\le 17$. That is, no situations were found such that $\pi(p^2)=\pi(p)$.
\end{ex}

Each polynomial in any of the previous families has a zero that is a fundamental unit, and we were able to use Lemma \ref{Lem:FU} to facilitate our investigation. However, in general, it is usually difficult to determine whether $f(x)$ has a zero that is a fundamental unit. Fortunately, in the situation when we have no knowledge of whether $f(x)$ has a zero that is a fundamental unit, we can use the following lemma, which is a special case of \cite[Corollary 5.2]{G}. 
\begin{lemma}\label{Lem:G}
  Let $f(x)\in \Z[x]$ be monic and irreducible over $\Q$. If $f(x)$ is not a cyclotomic polynomial, then there exist at most finitely many primes $p$ such that $f(x^p)$ is reducible over $\Q$.
\end{lemma}
We use Lemma \ref{Lem:G} to construct an infinite family for which Theorem \ref{Thm:Main1} applies.
\begin{ex}\label{NYJM}
It is shown in \cite{JonesNYJM} that there exist infinitely many primes $\rho$ such that
\[f(x)=x^6+(180\rho-1)x^3+1\]
is monogenic with $\Gal_{\Q}(f)\simeq D_6$, the dihedral group of order 12. Since $D_6$ contains an element of order 6, it follows from the Chebotarov density theorem that there exist infinitely many primes $p$ for which $f(x)$ is irreducible in $\F_p[x]$. Since $f(x)$ is clearly not a cyclotomic polynomial, we deduce from Lemma \ref{Lem:G} that there exist infinitely many primes $p$ such that $f(x)$ is $p$-irreducible. We say a pair of primes $(\rho,p)$ is \emph{viable} if $f(x)$ is monogenic and $p$-irreducible. Therefore, for such a viable pair $(\rho,p)$, we can apply Theorem \ref{Thm:Main1} to $f(x)$. A computer search verifies that out of the 121 viable pairs $(\rho,p)$, with $\rho,p\le 97$, exactly three viable pairs are such that $f(x^p)$ is not monogenic and for which $p$ is an $\Upsilon_f$-prime. That is, for these three pairs, we have that $\pi(p^2)=\pi(p)$, while $f(x)$ is monogenic with $\pi(p^2)=p\pi(p)$ for the remaining 118 viable pairs. These three exceptions are: \[(\rho,p,\pi(p^2))\in \{(5,5,18),(13,23,36),(67,47,144)\}.\]
\end{ex}

The following lemma is a special case of \cite[Lemma 3.1]{HJMJOU}.
\begin{lemma}\label{Lem:HJMJOU}
  Let $f(x)\in \Z[x]$ be monic and irreducible over $\Q$ with $\abs{f(0}\ge 2$. If $f(0)$ is squarefree, then $f(x^p)$ is irreducible over $\Q$ for all primes $p$.
\end{lemma}
\begin{ex}\label{Ex:GenQuadratics}
  Under certain restrictions on positive integers $a$ and $b$, quadratics of the form $f(x)=x^2-ax-b$ were investigated in \cite{JonesAJM}. However, the specific quadratics in Example \ref{Ex:Quadratics} were addressed in more generality in \cite{JonesTJM} than they were in \cite{JonesAJM}. On the other hand, the techniques used in \cite{JonesTJM} and \cite{JonesAJM} are nevertheless similar, and so we omit a complete discussion here of the quadratics of the form $f(x)=x^2-ax-b$. But we do provide a single example from \cite{JonesAJM}.
  Let \[f(x)=x^2-11x-43,\] which is easily seen to be monogenic. Observe that $f(x)$ is irreducible modulo 2. By Lemma \ref{Lem:HJMJOU}, $f(x^p)$ is irreducible over $\Q$ for all primes $p$. Thus, it follows that $f(x)$ is $p$-irreducible for infinitely many primes $p$. Hence, we can apply Theorem \ref{Thm:Main1} to $f(x)$. A computer determines that $f(x^p)$ is monogenic with $\pi(p^2)=p\pi(p)$ in $\Upsilon_f$ for all $p$-irreducible primes $p\le 313$, except $p\in \{2,5\}$. For these exceptions, we have that $\pi(2^2)=\pi(2)=3$ and $\pi(5^2)=\pi(5)=24$ in $\Upsilon_f$.
\end{ex}
For the construction of the next family, we apply Lemma \ref{Lem:HJMJOU} to a quadrinomial taken from \cite{JonesCM}.
\begin{ex}\label{Ex:CMquintics}
  Let
  \begin{equation}\label{Eq:FamCMquintics}
  f(x)=x^5+34x^4+9x^3+q,
  \end{equation} where $q$ is a prime such that $q\nmid 1120$, $q>44$ and
    \[\D:=(q+3542940)(3125q-56)\] is squarefree. Note that there exist infinitely many such primes $q$ (see \cite[Corollary 2.12]{HJMJOU}). It follows from \cite[Theorem 1.1]{JonesCM} (with $n=5$ and $m=1$) that $f(x)$ is irreducible over $\Q$, and that $f(x)$ is monogenic. Since $f(0)$ is squarefree with $f(0)>2$, it follows from Lemma \ref{Lem:HJMJOU} that $f(x^p)$ is irreducible over $\Q$ for all primes $p$. Observe that
    \begin{equation*}\label{Eq:5}
    f(x)\equiv x^5+x^3+1 \pmod{2},
    \end{equation*} which is irreducible in $\F_2[x]$, and consequently, $\Gal_{\Q}(f)$ contains an element of order 5. Thus, by the Chebotarev density theorem, there exist infinitely many primes $p$ for which $f(x)$ is irreducible modulo $p$. Consequently, there exist infinitely many primes $p$ such that $f(x)$ is $p$-irreducible, and Theorem \ref{Thm:Main1} applies in these situations. A computer search with $47\le q \le 97$
   and $p\le 17$ reveals that $f(x^p)$ is monogenic with $\pi(p^2)=p\pi(p)$, except when $p=2$ and $q\in \{47,59,67,71,79,83\}$. In these exceptional cases, we have that $\pi(2^2)=\pi(2)=31$.
\end{ex}



\end{document}